\documentclass[9pt,shortpaper,twoside,web]{ieeecolor}
\usepackage{generic}
\usepackage{cite}

\usepackage{amsmath,amssymb,amsfonts,amsthm}

\usepackage{graphics}
\usepackage{mathrsfs}
\usepackage{subfigure}
\usepackage{color} 
\usepackage{float}
\usepackage{graphicx}
\usepackage{multicol}
\usepackage{multirow}
\usepackage{cases}
\usepackage{bm}
\usepackage{textcomp}

\UseRawInputEncoding

\theoremstyle{definition} \newtheorem{rem}{Remark}
\theoremstyle{definition} 
\theoremstyle{definition} \newtheorem{thm}{Theorem}
\theoremstyle{definition} 
\theoremstyle{definition} \newtheorem{lem}{Lemma}
\theoremstyle{assumption} \newtheorem{ass}{Assumption}
\theoremstyle{definition} 
\theoremstyle{definition} 

\newcommand{\R}{\mathbb{R}}
\newcommand{\N}{\mathbf{N}}
\newcommand{\T}{\ensuremath{\textit{T}}}
\newcommand{\define}{:=}
\newcommand{\ve}[1]{\bm{#1}}

\DeclareMathOperator*{\col}{col}
\DeclareMathOperator*{\diag}{diag}

\usepackage{ifpdf}
\ifpdf
  \usepackage{epstopdf}
\fi

\makeatletter
\newif\if@restonecol
\makeatother

\usepackage[ruled,vlined]{algorithm2e}
\usepackage{algpseudocode}

\def\BibTeX{{\rm B\kern-.05em{\sc i\kern-.025em b}\kern-.08em
    T\kern-.1667em\lower.7ex\hbox{E}\kern-.125emX}}
\begin{document}
\title{Linear Convergence of Distributed Aggregative Optimization with Coupled Inequality Constraints}
\author{Kaixin Du and Min Meng
\thanks{This work was partially supported by Shanghai Pujiang Program under Grant 21PJ1413100, the National Natural Science Foundation of China under Grant 62103305, 61733018 and 62088101, and Shanghai Municipal Science and Technology Major Project under grant 2021SHZDZX0100 ({\em Corresponding author: Min Meng}).}
\thanks{K. Du is with the Shanghai Research Institute for Intelligent Autonomous Systems, Tongji University, Shanghai 200092, China. (e-mail: dukx@tongji.edu.cn).}
\thanks{M. Meng  is with the Department of Control Science and Engineering, College of Electronics and Information Engineering, and Shanghai Research Institute for Intelligent Autonomous Systems, Tongji University, Shanghai, China. (e-mail: mengmin@tongji.edu.cn).}
}

\maketitle

\begin{abstract}
This article investigates a distributed aggregative optimization problem subject to coupled affine inequality constraints, in which local objective  functions depend not only on their own decision variables but also on an aggregation of all the agents' variables. To our best knowledge, this work is the first to address this problem, and a novel distributed aggregative primal-dual algorithm is proposed based on the dual diffusion strategy and gradient tracking technique. Through rigorous analysis, it is shown that the devised algorithm converges to the optimal solution at a linear rate. Finally, a numerical example is conducted to illustrate the effectiveness of the theoretical results.
\end{abstract}


\section{Introduction}
\label{sec:introduction}
In recent years, a novel distributed optimization framework called distributed aggregative optimization has gained considerable attention due to its wide applications in different fields, such as optimal placement \cite{kansal2013optimal}, network congestion control \cite{barrera2014dynamic}, and network cost minimization problem \cite{cao2021distributed}. In this framework, a group of agents in a network cooperate to minimize the sum of local objective functions, which depend not only on their own decision variables but also on an aggregative variable consisting of the sum of functions of all agents' local decision variables.

The distributed aggregative optimization was originally proposed by Li \emph{et al.} \cite{Li2022Aggregative}. Subsequently, an online version with time-varying objective functions was studied in \cite{Li2022DistributedOnline}. The authors in \cite{Chen&LiangDistributed} attempted to consider communication with finite bites combined with quantization scheme. \cite{Carnevale2023nonconvex} investigated a distributed feedback optimization law for aggregative optimization problems in a nonconvex scenario. More recently, a distributed projection-free algorithm based on the Franke-Wolfe update in \cite{Wang2022Distributed} was designed to handle the feasible set. Considering heavy ball and Nesterov's accelerated methods, two algorithms called DAGT-HB and DAGT-NES in \cite{liu2023acc} were proposed to accelerate the convergence rate of distributed aggregative optimization. Recall that the works mentioned above considered either a simple set constraint or without any constraints. In addition, although stemming from distributed aggregative games \cite{Koshal2016Distributed,Liang2017Distributednash,gadjov2018passivity,yi2019operator}, it is noteworthy that in the aggregative optimization, agents aim to cooperatively find the optimal solution rather than to seek a Nash equilibrium (NE) or generalized NE (GNE) in a noncooperative manner.

Note that in many distributed scenarios in reality, the decisions of agents need to satisfy some global resource constraints \cite{Xu2019Dual,camisa2021distributed,falsone2017dual}, such as communication channel capacity and total energy \cite{heydaribeni2019distributed}, which could be modeled as coupled equality or inequality constraints. Therefore, it is highly significant to investigate distributed optimization with coupled constraints including equality or inequality constraints. A continuous primal-dual dynamics in \cite{guo2022exponential} was proposed which can be applied to distributed scenarios with coupled affine equality constraints, and theoretically shown to possess an exponential convergence. For discrete settings, in \cite{li2020distributed}, a distributed proximal primal-dual algorithm based on the proximal point method was studied for distributed optimization with a feasible set constraint and coupled inequality constraints, and rigorously proven to converge at a sublinear rate $O(1/\sqrt{k})$, where $k$ is the iteration number. Afterwards, a decentralized proximal primal-dual algorithm was devised in \cite{alghunaim2021dual} for minimizing the sum of local functions plus a coupled composite function, which was proven to have a linear convergence rate. In addition, it deserves to be mentioned that there are a number of other different algorithms for distributed optimization with coupled equality or inequality constraints, see \cite{Chang2015Multi-Agent,le2017neurodynamic,Li2023Proximal}, to just name a few. A recent study \cite{car2022distributed} introduced  an online projected aggregative tracking algorithm to handle time-varying objective functions, aggregation rules and inequality constraints. Furthermore, the objective error was proven to converge at a linear rate in the static case. Treated as a cooperative formulation for multi-agent decision problems, authors in \cite{huang2021primal} addressed cooperative aggregative games with set constraints as well as coupled affine inequality constraints, and developed a primal decomposition approach based on the Douglas-Rachford splitting method for finding GNE. Even though the convergence of the proposed algorithm was established, the linear convergence has not been established..

It can be carefully observed that the distributed aggregative optimization problem with coupled constraints has not yet been considered. Inspired by the practical applications, this article investigates distributed aggregative optimization subject to coupled affine inequality constraints, in which each local objective function hinges on not only an associated local decision variable but also an aggregative variable. The latter is obtained by summing functions of local decision variables, and thus is possibly not available to each agent in the distributed setup. In fact, each agent is only able to access partial information about its own decision variable, cost and constraint functions. Hence, local information exchanges through a communication network between neighboring agents is required to reach the global optimality. Then a novel primal-dual algorithm, where each agent is equipped with a global dual variable along with the related local estimates and the trackers are utilized to handle the gradient average, as well as the aggregation variable, is devised to deal with the formulated problem in this paper and is also elaborately proven to converge linearly. The main contributions of this article can be summarized as follows:
\begin{itemize}
  \item [i)] To the best of our knowledge, this work is the first to investigate distributed aggregative optimization subject to coupled inequality constraints. The proposed distributed aggregative primal-dual algorithm makes a slight but critical modification to the typical primal-dual algorithm and draws into a positive semi-definite matrix associated with the communication topology, and meanwhile adheres to the dual diffusion strategy and gradient tracking technique.
  \item [ii)] Under mild assumptions, the global linear convergence is established and an explicit upper bound on the rate of convergence is provided. Additionally, the conditions that stepsizes are required to satisfy are also given, which are related with the communication structure, the affine inequality, the strong convexity and smoothness parameters of the objective function.
\end{itemize}

The rest of this article is organized as follows. Section \ref{sec2} formulates the problem. Section \ref{sec3} presents the developed algorithm and convergence results. The performance of the designed algorithm is illustrated in Section \ref{sec4}. Finally, Section \ref{sec5} makes a brief conclusion and future vision.

\emph{Notations:} Denote by $\R^n$, $\R^n_+$ and $\R^{m\times n}$ $n$-dimensional real vector space, $n$-dimensional vector set with every entry nonnegative, and the set of $m\times n$ real matrices, respectively. $I_N\in\R^{N\times N}$ denotes the identity matrix of size $N$. $\ve{1}_N\in\R^N$ is a vector with all entries equal to one, and $\ve{0}$ is a  vector of compatible dimension with all entries being $0$. $[n]$ represents the index set $\{1,\ldots,n\}$ for an integer $n>0$. $\diag\{A_1,\ldots,A_n\}$ stands for a block diagonal matrix with diagonal blocks $A_i$, $i\in[n]$.  $A^{\T}$ is denoted as the transpose of matrix $A$. Let $\col\{x_1,\ldots,x_N\}:=(x_1^{\T},\ldots,x_N^{\T})^{\T}$. $\otimes$ is the Kronecker product. For a square matrix $A$, denote by $\lambda_{\min}(A)$ and $\lambda_{\max}(A)$ the smallest and largest eigenvalues, respectively. $\underline{\sigma}(A)$ denotes the smallest nonzero singular value of any matrix $A$. Let $\Pi_X[x]={\arg\min}_{y\in X}\|y-x\|^2$ be the projection of a vector $x\in\R^n$ onto a closed convex set $X\subset\R^n$ and $\N_X(x)=\{s:s^{\T}(y-x)\leq 0,\ \forall y\in X$\} be the normal cone of $x\in X$. For a positive semi-definite $Q\in\R^{n\times n}$ and $x\in\R^n$, let $\|x\|_Q^2:=x^{\T}Qx$. For a function $g:\R^d\to\R^n$ with component functions $g_1,\ldots,g_n$, denote $\nabla g(x)\define [\nabla g_1(x),\ldots,\nabla g_n(x)]$.

\section{PROBLEM FORMULATION}\label{sec2}

 Consider a network of $N$ agents. The goal is to solve the following aggregative optimization problem subject to coupled inequality constraints:
\begin{equation}\label{prob}
\begin{aligned}
  &\min_{\ve{x}\in\R^d}~f(\ve{x})=\sum_{i=1}^N f_i(x_i,\phi(\ve{x}))\\
  &\text{s.t.}~\sum_{i=1}^N A_ix_i\leq\sum_{i=1}^Nb_i,
\end{aligned}
\end{equation}
where $\ve{x}=\col\{x_1,\ldots,x_N\}$ is the global decision variable with $x_i\in\R^{d_i}$ being the local decision variable. The function $f_i:\R^{d_i}\times\R^n\to\R$, the matrix $A_i\in\R^{m\times d_i}$ and $b_i\in \R^m$ are privately known to agent $i$, where $d=\sum_{i=1}^Nd_i$, while the aggregation function $\phi(\ve{x})$ has the form
\begin{equation}\label{aggr func}
  \phi(\ve{x}):=\frac{\sum_{i=1}^N h_i(x_i)}{N},
\end{equation}
in which $h_i:\R^{d_i}\to\R^n$ modeling the contribution of the corresponding decision $x_i$ to $\phi(\ve{x})$ is only accessible to agent $i$.

To proceed, for brevity, for any $i\in [N]$, denote $\nabla_1f_i(x_i,z_i):=\frac{\partial f_i(x_i,z_i)}{\partial x_i}$ and $\nabla_2f_i(x_i,z_i):=\frac{\partial f_i(x_i,z_i)}{\partial z_i}$. For $\ve{x}=\col\{x_1,\ldots,x_N\}\in\R^d$ and $\ve{z}=\col\{z_1,\ldots,z_N\}\in\R^{Nn}$, define $\nabla_1 f(\ve{x},\ve{z}):=\col\{\nabla_1 f_1(x_1,z_1)\ldots,\nabla_1 f_N(x_N,z_N)\}$, $\nabla_2 f(\ve{x},\ve{z}):=\col\{\nabla_2 f_1(x_1,z_1)\ldots,\nabla_2 f_N(x_N,z_N)\}$. Moreover, let $h(\ve{x}):=\col\{h_1(x_1),\ldots,h_N(x_N)\}$, $A:=[A_1,\ldots,A_N]$, $b:=\sum_{i=1}^Nb_i$.

For problem (\ref{prob}), introduce the Lagrange function
\begin{align}\label{fun:lagrange}
  L(\ve{x},\lambda)=f(\ve{x})+\lambda^{\T}(A\ve{x}-b),
\end{align}
where $\lambda\in\R^{m}$ is the corresponding Lagrange multiplier. It is known that if $\ve{x}^*=\col\{x_1^*,\ldots,x_N^*\}$ is the optimal solution to problem (\ref{prob}), there exists $\lambda^*$ such that $(\ve{x}^*,\lambda^*)$ is the Karush-Kuhn-Tucker (KKT) point of this problem satisfying
\begin{align}
  &\nabla_1f_i(x_i^*,\phi(\ve{x}^*))+\frac{\nabla h_i(x_i^*)}{N}\sum_{j=1}^N\nabla_2f_j(x_j^*,\phi(\ve{x}^*))\notag\\
  &+A_i^{\T}\lambda^*=\ve{0},\ \forall i\in[N],\label{KKT-x}\\
  &A\ve{x}^*-b\in\N_{\R_+^m}(\lambda^*)\label{KKT-dual}.
\end{align}

Next, we postulate some standard assumptions to be used for solving problem (\ref{prob}).
\begin{ass}\label{ass:function} The following hold for problem (\ref{prob}):
\begin{itemize}
  \item [i)] The global cost function $f$ is differentiable and $\nu$-strongly convex, i.e., for all $\ve{x},\ve{x}'\in\R^d$,
   \begin{align*}
    (\nabla f(\ve{x})-\nabla f(\ve{x}'))^{\T}(\ve{x}-\ve{x})&\geq\nu\|\ve{x}-\ve{x}'\|^2.
   \end{align*}
   \item [ii)] Define $\nabla h(\ve{x})\define\diag\{\nabla h_1(x_1),\ldots, \nabla h_N(x_N)\}$.  $\nabla_1f(\ve{x},\ve{z})+\nabla h(\ve{x})(\ve{1}_N\otimes\frac{1}{N} \sum_{i=1}^N\nabla_2f_i(x_i,z_i))$ is $L_1$- Lipschitz, i.e., for all $\ve{x},\ve{x}'\in\R^d$ and $\ve{z},\ve{z}'\in\R^{Nn}$,
   \begin{align*}\label{ass:Lip}
     &\Big\|\nabla_1f(\ve{x},\ve{z})+\nabla h(\ve{x})\Big(\ve{1}_N\otimes\frac{1}{N} \sum_{i=1}^N\nabla_2f_i(x_i,z_i)\Big)\notag\\
     &-\nabla_1f(\ve{x}',\ve{z}')-\nabla h(\ve{x}')\Big(\ve{1}_N\otimes\frac{1}{N} \sum_{i=1}^N\nabla_2f_i(x_i',z_i')\Big)\Big\|\notag\\
     \leq& L_1(\|\ve{x}-\ve{x'}\|+\|\ve{z}-\ve{z'}\|).
     \end{align*}
  \item [iii)] $\nabla_2f(\ve{x},\ve{z})$ is $L_2$-Lipschitz continuous, i.e., for all $\ve{x},\ve{x}'\in\R^d$ and $\ve{z},\ve{z}'\in\R^{Nn}$,
  \begin{equation*}
    \|\nabla_2f(\ve{x},\ve{z})-\nabla_2f(\ve{x'},\ve{z'})\|\leq L_2(\|\ve{x}-\ve{x'}\|+\|\ve{z}-\ve{z'}\|).
  \end{equation*}
  \item [iv)] For any $i\in [N]$, $h_i$ is differentiable and $\nabla h_i$ is bounded, i.e., there exists a constant $L_3>0$ such that $\|\nabla h_i(\cdot)\|\leq L_3$.
  \item [v)] For any $i\in[N]$, $A_i$ has full row rank.
\end{itemize}
\end{ass}

\begin{rem}
It is worth noting that while the global objective function $f$ is assumed to be strongly convex, each local objective function $f_i$, $i\in[N]$ may not be strongly convex or even convex. In fact, the strong convexity of $f$ is typically required to guarantee linear convergence in distributed algorithms with fixed step-sizes. Assumption \ref{ass:function}-ii) and \ref{ass:function}-iii) are commonly used in the analysis of distributed aggregative optimization \cite{Li2022DistributedOnline,Chen&LiangDistributed}. Assumption \ref{ass:function}-iv) implies $\nabla h_i$ is $L_3$-Lipschitz. Under Assumption \ref{ass:function}-v), $A_iA_i^{\T}$ is positive definite. Assumption \ref{ass:function}-v) is common and essential to design linearly convergent algorithms when affine coupled constraints exist \cite{alghunaim2020proximal,meng2022linear}.
\end{rem}

The agents communicate with each other over a network described by an undirected graph $\mathcal{G}=(\mathcal{V},\mathcal{E},\mathcal{W})$, where $\mathcal{V}=[N]$ is the node set with node $i$ corresponding to the $i$-th agent, $\mathcal{E}\subset\mathcal{V}\times\mathcal{V}$ is the edge set, and $\mathcal{W}=[w_{ij}]\in\mathbb{R}^{N\times N}$ is the adjacency matrix of $\mathcal{G}$. Here, $w_{ij}>0$ if $(j,i)\in\mathcal{E}$, i.e., agent $i$ and $j$ are able to communicate with each other directly, and $w_{ij}=0$, otherwise.

\begin{ass}\label{ass:network} The graph $\mathcal{G}$ is undirected and connected. The matrix $\mathcal{W}$ is symmetric and doubly stochastic, i.e., $\mathcal{W}=\mathcal{W}^{\T}$ and $\mathcal{W}\ve{1}_N=\ve{1}_N$.
\end{ass}

Under Assumption \ref{ass:network}, the matrix $\mathcal{W}$ is primitive. Hence, it holds from the Perron-Frobenius theorem \cite{horn2012MA} that $\mathcal{W}$ has a simple eigenvalue $1$ with all the other eigenvalues belonging to $(-1,1)$. For any $\ve{u}\in\R^{Nm}$, $(I_{Nm}-\mathcal{W}\otimes I_m)\ve{u}=\ve{0}$ if and only if $\ve{u}=\ve{1}_N\otimes u$ for some $u\in\R^m$. Define $B^2$ as follows:
$$B^2:=\frac{1}{2}(I_N-\mathcal{W}),$$
then $B^2$ is symmetric and positive semi-definite with eigenvalues in $[0,1)$, and it holds that $(B\otimes I_m)\ve{u}=\ve{0}$ if and only if $\ve{u}=\ve{1}_N\otimes u$ for some $u\in\R^m$. Moreover, it can be observed that the matrix $B^2$ reveals the network structure, which is of utmost importance for introducing distributed algorithms to solve the formulated problem, while $B$ may not.

\begin{lem}\cite{horn2012MA}\label{lem:network}
Suppose Assumption \ref{ass:network} is satisfied, then
\begin{itemize}
  \item [i)] $(\frac{1}{N}\ve{1}_N\ve{1}_N^{\T}\otimes I_n)(\mathcal{W}\otimes I_n)=(\mathcal{W}\otimes I_n)(\frac{1}{N}\ve{1}_N\ve{1}_N^{\T}\otimes I_n)=\frac{1}{N}\ve{1}_N\ve{1}_N^{\T}\otimes I_n$.
  \item [ii)] $\rho:=\|\mathcal{W}-\frac{1}{N}\ve{1}_N\ve{1}_N^{\T}\|< 1$.
  \item [iii)] $\|\mathcal{W}-I_N\|\leq 2$.
  \item [iv)] For $\ve{z}\in\R^{Nn}$, $\|(\mathcal{W}\otimes I_n)\ve{z}-(\frac{1}{N}\ve{1}_N\ve{1}_N^{\T}\otimes I_n)\ve{z}\|\leq\rho\|\ve{z}-(\frac{1}{N}\ve{1}_N\ve{1}_N^{\T}\otimes I_n)\ve{z}\|$.
\end{itemize}
\end{lem}

\section{MAIN RESULTS}\label{sec3}

In this section, we first propose a novel distributed aggregative primal-dual algorithm to find the optimal solution for the coupled constrained aggregative optimization problem (\ref{prob}) and then show the linear convergence result.

At iteration $k$, each agent $i$, $i\in[N]$ is equipped with variables $x_{i,k}$ and $\lambda_{i,k}$ to estimate the optimal solution and global optimal dual variable, respectively, while $v_{i,k}$ and $y_{i,k}$ are the auxiliary variables. $z_{i,k}$ is exploited to track the average $\phi(\ve{x}_k)$ and $\mu_{i,k}$ is leveraged to track the average gradient $\frac{1}{N}\sum_{i=1}^n\nabla_2f_i(x_{i,k},\phi(\ve{x}_k))$. Define $\Lambda:=\diag\{A_1,\ldots,A_N\}$, $\ve{b}:=\col\{b_1,\ldots,b_N\}$ and $\mathcal{B}:=B\otimes I_m$. By combining an augmented Lagrange method with the distributed gradient tracking scheme \cite{pu2018}, for the time being, we obtain a distributed aggregative primal-dual algorithm to handle problem (\ref{prob}) as follows:
\begin{subequations}\label{alg:vec}
  \begin{align}
  \ve{x}_{k+1}=&\ve{x}_k-\alpha(\nabla_1f(\ve{x}_k,\ve{z}_k)+\nabla h(\ve{x}_k)\ve{\mu}_k+\Lambda^{\T}\ve{\lambda}_k),\label{alg:vec-1}\\
  \ve{z}_{k+1}=&(\mathcal{W}\otimes I_n)\ve{z}_k+h(\ve{x}_{k+1})-h(\ve{x}_k),\label{alg:vec-2}\\
  \ve{\mu}_{k+1}=&(\mathcal{W}\otimes I_n)\ve{\mu}_k+\nabla_2f(\ve{x}_{k+1},\ve{z}_{k+1})-\nabla_2f(\ve{x}_k,\ve{z}_k),\label{alg:vec-3}\\
  \ve{v}_{k+1}=&\ve{\lambda}_k-\mathcal{B}^2\ve{\lambda}_k+\beta(\Lambda \ve{x}_{k+1}-\ve{b})+\mathcal{B}\ve{y}_k,\label{alg:vec-4}\\
  \ve{y}_{k+1}=&\ve{y}_k-\gamma\mathcal{B}\ve{v}_{k+1},\label{alg:vec-5}\\
  \ve{\lambda}_{k+1}=&\Pi_{\R^{Nm}_+}[\ve{v}_{k+1}],\label{alg:vec-6}
\end{align}
\end{subequations}
where $\alpha$, $\beta$, $\gamma>0$ are stepsizes, which will be specified in the convergence analysis later on. Different from the conventional primal-dual method, the auxiliary variable $\ve{v}_k$ instead of the dual variable $\ve{\lambda}_k$ is utilized in the dual update $(\ref{alg:vec-5})$. In addition, the addition term $-\mathcal{B}^2\ve{\lambda}_k$ in $(\ref{alg:vec-4})$ appears by adding $-\frac{1}{2}\ve{\lambda}^{\T}\mathcal{B}^2\ve{\lambda}$ to the Lagrange function (\ref{fun:lagrange}) due to $\mathcal{B}\ve{\lambda}^*=\ve{0}$. The two slight differences increase the complexity of the algorithm, which is rather significant to establish linear convergence. Iteration $(\ref{alg:vec-6})$ uses projection onto $\R^{Nm}_+$ to deal with the inequality constraints.

It can be seen that algorithm (\ref{alg:vec}) cannot be directly implemented in a fully distributed manner due to matrix $\mathcal{B}$ is involved. Taking $\mathcal{B}^2$ into account, next, we transfer (\ref{alg:vec}) into a fully distributed and equivalent form. To this end, multiplying (\ref{alg:vec-5}) by $\mathcal{B}$ and substituting it into (\ref{alg:vec-4}), it gives that
\begin{align*}
  \ve{v}_{k+1}=&(I_{Nm}-\gamma\mathcal{B}^2)\ve{v}_k-(I_{Nm}-\mathcal{B}^2)(\ve{\lambda}_k-\ve{\lambda}_{k-1})\notag\\
  &+\beta\Lambda(\ve{x}_{k+1}-\ve{x}_k).
\end{align*}
Therefore, iteration (\ref{alg:vec}) can be reformulated into a fully distributed manner as Algorithm \ref{alg:primaldual}.

\begin{algorithm}[!hpt]\caption{Distributed Aggregative Primal-Dual Algorithm}\label{alg:primaldual}
At iteration $k$, each agent $i$, $i\in[N]$ maintains its own variables $x_{i,k}\in\R^{d_i}$, $z_{i,k}\in\R^n$, $\mu_{i,k}\in\R^n$, $v_{i,k}\in\R^m$ and $\lambda_{i,k}\in\R^m$. Let $C=[c_{ij}]_{N\times N}:=B^2=\frac{I_N-\mathcal{W}}{2}$.

 {\bf Initialization:} For any $i\in[N]$, initialize $x_{i,0}$ and $\lambda_{i,0}$ arbitrarily, and set $z_{i,0}=h_i(x_{i,0})$, $\mu_{i,0}=\nabla_2f_i(x_{i,0},z_{i,0})$,  $v_{i,0}=\ve{0}_m$ and $\lambda_{i,-1}=\ve{0}_m$.

{\bf Iteration:} For $k\geq 0$, each agent $i$ performs the following updates:
\begin{subequations}\label{alg:dist}
  \begin{align}
  x_{i,k+1}=&x_{i,k}-\alpha(\nabla_1f_i(x_{i,k},z_{i,k})+\nabla h_i(x_{i,k})\mu_{i,k}\notag\\
  &+A_i^{\T}\lambda_{i,k}),\label{alg:dist-1}\\
  z_{i,k+1}=&\sum_{j=1}^Nw_{ij}z_{j,k}+h_i(x_{i,k+1})-h_i(x_{i,k}),\label{alg:dist-2}\\
 \mu_{i,k+1}=&\sum_{j=1}^Nw_{ij}\mu_{j,k}+\nabla_2f_i(x_{i,k+1},z_{i,k+1})\notag\\
 &-\nabla_2f_i(x_{i,k},z_{i,k}),\label{alg:dist-3}\\
  v_{i,k+1}=&v_{i,k}-\gamma\sum_{j=1}^Nc_{ij}v_{j,k}-\sum_{j=1}^Nc_{ij}(\lambda_{j,k}-\lambda_{j,k-1})\notag\\
  &+ \lambda_{i,k}-\lambda_{i,k-1}+\beta A_i(x_{i,k+1}-x_{i,k}),\label{alg:dist-4}\\
 \lambda_{i,k+1}=&\Pi_{\R^m_+}[v_{i,k+1}],\label{alg:dist-5}
\end{align}
where $\alpha$, $\beta$, $\gamma>0$ are stepsizes to be determined.
\end{subequations}
\end{algorithm}

\begin{rem}
The presence of the aggregative structure in each local objective function makes the problem studied in this article distinctive from that of \cite{alghunaim2021dual,meng2022linear}, and thus there is a substantial difference in theoretical analysis owing to the gradient tracking processes utilized to deal with the aggregation variable and gradient average.
\end{rem}

From now on, we focus on establishing the linear convergence of algorithm (\ref{alg:vec}) which is equivalent to Algorithm \ref{alg:primaldual}. Let us start by presenting some useful lemmas. First, Lemma \ref{lem:average} gives some observations on  variables $\ve{z}_k$ and $\ve{\mu}_k$.
\begin{lem}\label{lem:average}
  Under Assumption \ref{ass:network}, the following equalities hold:
  \begin{align}
    \bar{z}_k&:=\frac{1}{N}\sum_{i=1}^Nz_{i,k}=\frac{1}{N}\sum_{i=1}^Nh_i(x_{i,k})=\phi(\ve{x}_k),\label{lem:average-1}\\
    \bar{\mu}_k&:=\frac{1}{N}\sum_{i=1}^N\mu_{i,k}=\frac{1}{N}\sum_{i=1}^N\nabla_2f_i(x_{i,k},z_{i,k}).\label{lem:average-2}
  \end{align}
\end{lem}

\begin{proof}
  See Appendix \ref{proof:lem:average}.
\end{proof}

We show the existence and properties of the fixed points of iteration (\ref{alg:vec}), which are summarized as the following lemma.

\begin{lem}\label{lem:fix}
  Iteration (\ref{alg:vec}) has a fixed point $(\ve{x}^*,\ve{z}^*,\ve{\mu}^*,\ve{v}^*,\ve{y}^*,\ve{\lambda}^*)$ satisfying
  \begin{subequations}\label{lem:fixed-pont}
  \begin{align}
  \ve{0}&=\nabla_1f(\ve{x}^*,\ve{z}^*)+\nabla h(\ve{x}^*)\ve{\mu}^*+\Lambda^{\T}\ve{\lambda}^*,\label{lem:fixed-pont-1}\\
  \ve{z}^*&=(\mathcal{W}\otimes I_n)\ve{z}^*,\label{lem:fixed-pont-2}\\
  \ve{\mu}^*&=(\mathcal{W}\otimes I_n)\ve{\mu}^*,\label{lem:fixed-pont-3}\\
  \ve{v}^*&=\ve{\lambda}^*+\beta(\Lambda \ve{x}^*-\ve{b})+\mathcal{B}\ve{y}^*,\label{lem:fixed-pont-4}\\
  \ve{\lambda}^*&=\Pi_{\R^{Nm}_+}[\ve{v}^*],\label{lem:fixed-pont-5}\\
  \ve{0}&=\mathcal{B}\ve{v}^*,\label{lem:fixed-pont-6}
\end{align}
\end{subequations}
and $\mathcal{B}^2\ve{\lambda}^*=\ve{0}$. Moreover, for any fixed point $(\ve{x}^*,\ve{z}^*,\ve{\mu}^*,\ve{v}^*,\ve{y}^*,\ve{\lambda}^*)$, it holds that $\ve{z}^*=\ve{1}_N\otimes \phi(\ve{x}^*)$, $\ve{\mu}^*=\ve{1}_N\otimes\frac{1}{N}\sum_{i=1}^N\nabla_2f_i(x_i^*,\phi(\ve{x}^*))$, and $\ve{\lambda}^*=\ve{1}_N\otimes\lambda^*$ with $\lambda^*\in \R^m$, and $(\ve{x}^*,\lambda^*)$ is a KKT point for problem (\ref{prob}). Therefore,  $\ve{x}^*$ is the optimal solution to problem (\ref{prob}), and $\lambda^*$ is the optimal global dual variable.
\end{lem}

\begin{proof}
See Appendix \ref{proof:lem:fix}.
\end{proof}

\begin{rem}
It is obtained based on the strong convexity of $f$ that the optimal solution $\ve{x}^*$ is unique and the optimal dual variable $\lambda^*$ is unique since matrix $A$ has full rank. However, the uniqueness of $\ve{y}^*$ cannot be ensured as $\ve{y}^*$ is chosen from the range space of $\mathcal{B}$.
\end{rem}

For the fixed point $\ve{x}^*$, $\ve{\lambda}^*=\ve{1}_N\otimes \lambda^*$, $\ve{v}^*=\ve{1}_N\otimes v^*$ and $\ve{y}^*$ belonging to the range space of $\mathcal{B}$,  define $\tilde{\ve{x}}_k:=\ve{x}_k-\ve{x}^*$, 
$\tilde{\ve{v}}_k:=\ve{v}_k-\ve{v}^*$, $\tilde{\ve{\lambda}}_k:=\ve{\lambda}_k-\ve{\lambda}^*$, $\tilde{\ve{y}}_k:=\ve{y}_k-\ve{y}^*$, it derives from (\ref{alg:vec}) and Lemma \ref{lem:fix} that
\begin{subequations}\label{alg:vec-error}
  \begin{align}
  \tilde{\ve{x}}_{k+1} = &\tilde{\ve{x}}_k-\alpha\Lambda^{\T}\tilde{\ve{\lambda}}_k-\alpha\Big(\nabla_1f(\ve{x}_k,\ve{z}_k)+\nabla h(\ve{x}_k)\ve{\mu}_k\notag\\
  &-\nabla_1f(\ve{x}^*,(\ve{1}_N\otimes \phi(\ve{x}^*))\notag\\
  &-\nabla h(\ve{x}^*)\Big(\ve{1}_N\otimes\frac{1}{N} \sum_{i=1}^N\nabla_2f_i(x_i^*,\phi(\ve{x}^*))\Big)\Big),\label{alg:vec-error-1}\\
  \tilde{\ve{v}}_{k+1} = &\tilde{\ve{\lambda}}_k-\mathcal{B}^2\tilde{\ve{\lambda}}_k+\beta\Lambda \tilde{\ve{x}}_{k+1}+\mathcal{B}\tilde{\ve{y}}_k\label{alg:vec-error-4},\\
  \tilde{\ve{\lambda}}_{k+1}= &\Pi_{\R^{Nm}_+}[\ve{v}_{k+1}]-\Pi_{\R^{Nm}_+}[\ve{v}^*],\label{alg:vec-error-5}\\
  \tilde{\ve{y}}_{k+1} = &\tilde{\ve{y}}_k-\gamma\mathcal{B}\tilde{\ve{v}}_{k+1}.\label{alg:vec-error-6}
\end{align}
\end{subequations}

Next, we give a crucial lemma which characterizes the recursive inequality relations on error variables $\|\tilde{\ve{x}}_k\|$, $\|\tilde{\ve{\lambda}}_k\|$, $\|\tilde{\ve{y}}_k\|$, $\|\ve{z}_k-\ve{1}_N\otimes\bar{z}_k\|$ and $\|\ve{\mu}_k-\ve{1}_N\otimes\bar{\mu}_k\|$.

\begin{lem}\label{lem:error-iter}
  Suppose Assumptions \ref{ass:function} and \ref{ass:network} are satisfied, then the errors $\|\tilde{\ve{x}}_k\|$, $\|\tilde{\ve{\lambda}}_k\|$, $\|\tilde{\ve{y}}_k\|$, $\|\ve{z}_k-\ve{1}_N\otimes\bar{z}_k\|$ and $\|\ve{\mu}_k-\ve{1}_N\otimes\bar{\mu}_k\|$ generated by Algorithm \ref{alg:primaldual} satisfy
  \begin{align}
    &\|\tilde{\ve{x}}_{k+1}\|^2\notag\\
    \leq&\Big(1-\alpha\Big(\frac{\nu}{2}-4\alpha L_1^2(1+2L_3^2)\Big)\Big)\|\tilde{\ve{x}}_k\|^2-\alpha\nu\|\tilde{\ve{x}}_k\|^2\notag\\
    &-\alpha^2\|\Lambda^{\T}\tilde{\ve{\lambda}}_k\|^2
    -2\alpha\tilde{\ve{\lambda}}_k^{\T}\Lambda\tilde{\ve{x}}_{k+1}\notag\\
    &+4\Big(2\alpha^2+\frac{\alpha}{\nu}\Big)L_1^2\|\ve{z}_k-\ve{1}_N\otimes\bar{z}_k\|^2\notag\\
    &+2\Big(\alpha^2+\frac{2\alpha}{\nu}\Big)L_3^2\|\ve{\mu}_k-\ve{1}_N\otimes\bar{\mu}_k\|^2,\label{error-x}
  \end{align}
  \begin{align}
    &\|\tilde{\ve{\lambda}}_{k+1}\|_{I_{Nm}-\gamma\mathcal{B}^2}^2+\frac{1}{\gamma}\|\tilde{\ve{y}}_{k+1}\|^2\notag\\
    \leq&\|\tilde{\ve{\lambda}}_k\|^2
    +2\|\mathcal{B}^2\tilde{\ve{\lambda}}_k\|^2+2\beta^2\|\Lambda\tilde{\ve{x}}_{k+1}\|^2
    -2\tilde{\ve{\lambda}}_k^{\T}\mathcal{B}^2\tilde{\ve{\lambda}}_k\notag\\
    &+2\beta\tilde{\ve{\lambda}}_k^{\T}\Lambda\tilde{\ve{x}}_{k+1}-\|\mathcal{B}\tilde{\ve{y}}_k\|^2+\frac{1}{\gamma}\|\tilde{\ve{y}}_k\|^2,\label{error-dual}
  \end{align}
   where $\gamma>0$ is selected as $\gamma<\frac{1}{\underline{\sigma}^2(\mathcal{B})}$,
  \begin{align}
    &\|\ve{z}_{k+1}-\ve{1}_N\otimes\bar{z}_{k+1}\|^2\notag\\
    \leq&\Big(\frac{1+\rho^2}{2}+16\frac{1+\rho^2}{1-\rho^2}L_1^2L_3^2\alpha^2\Big)\|\ve{z}_k-\ve{1}_N\otimes\bar{z}_k\|^2\notag\\
    &+8\frac{1+\rho^2}{1-\rho^2}L_1^2L_3^2(1+2L_3^2)\alpha^2\|\tilde{\ve{x}}_k\|^2+2\frac{1+\rho^2}{1-\rho^2}L_3^2\alpha^2\|\Lambda^{\T}\tilde{\ve{\lambda}}_k\|^2\notag\\
  &+4\frac{1+\rho^2}{1-\rho^2}L_3^4\alpha^2\|\ve{\mu}_k-\ve{1}_N\otimes\bar{\mu}_k\|^2,\label{error-aggre}
  \end{align}
  \begin{align}
    &\|\ve{\mu}_{k+1}-\ve{1}_N\otimes\bar{\mu}_{k+1}\|^2\notag\\
    \leq&\Big(\frac{1+\rho^2}{2}+8\frac{1+\rho^2}{1-\rho^2}L_2^2L_3^2(1+2L_3^2)\alpha^2\Big)\|\ve{\mu}_k-\ve{1}_N\otimes\bar{\mu}_k\|^2\notag\\
    &+16\frac{1+\rho^2}{1-\rho^2}L_1^2L_2^2(1+2L_3^2)^2\alpha^2\|\tilde{\ve{x}}_k\|^2\notag\\
  &+16\frac{1+\rho^2}{1-\rho^2}(2L_1^2L_2^2(1+2L_3^2)\alpha^2+L_2^2)\|\ve{z}_k-\ve{1}_N\otimes\bar{\ve{z}}_k\|^2\notag\\
   &+4\frac{1+\rho^2}{1-\rho^2}L_2^2(1+2L_3^2)\alpha^2\|\Lambda^{\T}\tilde{\ve{\lambda}}_k\|^2.\label{error-nabla2}
  \end{align}
\end{lem}

\begin{proof}
See Appendix \ref{proof:lem:error-iter}.
\end{proof}

At present, it is ready to present our main result of this article.

\begin{thm}\label{thm:linearrate}
  Suppose Assumptions \ref{ass:function} and \ref{ass:network} are satisfied. Let the sequences $\{\ve{x}_k\}$, $\{\ve{\lambda}_k\}$, $\{\ve{y}_k\}$, $\{\ve{z}_k\}$ and $\{\ve{\mu}_k\}$ be generated by Algorithm \ref{alg:primaldual}. Define
  \begin{align*}
  \varepsilon_k:=&\|\tilde{\ve{x}}_k\|_{I_d-2\alpha\beta\Lambda^{\T}\Lambda}^2+\frac{\alpha}{\beta}\|\tilde{\ve{\lambda}}_k\|_{I_{Nm}-\gamma\mathcal{B}^2}^2
  +\frac{\alpha}{\beta\gamma}\|\tilde{\ve{y}}_k\|^2\notag\\
  +&\frac{1-\rho^2}{4L_3^2(1+\rho^2)}\|\ve{z}_k-\ve{1}_N\otimes\bar{z}_k\|^2\notag\\
  +&\frac{(1-\rho^2)\alpha}{(1+\rho^2)\gamma}\|\ve{\mu}_k-\ve{1}_N\otimes\bar{\mu}_k\|^2.
   \end{align*}
If the stepsizes $\alpha$, $\beta$, $\gamma>0$ are chosen such that
\begin{align}\label{eq:stepsize}
  \kappa:=&1-\alpha\Big(\frac{\nu}{2}-4\alpha L_1^2(1+2L_3^2)\Big(\frac{3}{2}+4L_2^2(1+2L_3^2)\frac{\alpha}{\gamma}\Big)\Big)<1,\notag\\
  \kappa_1:=&\frac{1+\rho^2}{2}+16\frac{1+\rho^2}{1-\rho^2}L_1^2L_3^2(3\alpha^2+\frac{\alpha}{\nu}\Big)\notag\\
  &+\frac{64L_3^2\alpha}{\gamma}\frac{1+\rho^2}{1-\rho^2}(2L_1^2L_2^2(1+2L_3^2)\alpha^2+L_2^2)<1,\notag\\
  \kappa_2:=&\frac{1+\rho^2}{2}+8\frac{1+\rho^2}{1-\rho^2}L_2^2L_3^2(1+2L_3^2)\alpha^2\notag\\
  &+2\frac{1+\rho^2}{1-\rho^2}L_3^2\Big(\alpha+\frac{2}{\nu}\Big)\gamma+\gamma\alpha L_3^2\frac{1+\rho^2}{1-\rho^2}< 1,\notag\\
  \beta<&\min\Big\{\frac{\nu}{2\kappa\lambda_{\max}(\Lambda^{\T}\Lambda)},\frac{1}{\alpha c_1\lambda_{\min}(\Lambda\Lambda^{\T})}\Big\},\notag\\
  \gamma<&\min\Big\{\frac{2-2\lambda_{\max}^2(\mathcal{B})}{1-\alpha\beta c_1\lambda_{\min}(\Lambda\Lambda^{\T})},\frac{1}{\underline{\sigma}^2(\mathcal{B})}\Big\},
 \end{align}
with $c_1:=\frac{1}{2}-4L_2^2(1+2L_3^2)\frac{\alpha}{\gamma}$, then $\varepsilon_k$ converges to zero at a linear rate. Specifically,
  \begin{align}\label{thm:rate}
    \varepsilon_{k+1}\leq \tau \varepsilon_k,
  \end{align}
  where $\tau:=\max\{\kappa,\kappa_1,\kappa_2,1-\alpha\beta c_1\lambda_{\min}(\Lambda\Lambda^{\T}),1-\gamma\underline{\sigma}^2(\mathcal{B})\}\in(0,1)$.
\end{thm}

\begin{proof}
After a simple calculation, one has from Lemma \ref{lem:error-iter} that
\begin{align}\label{eq:1}
  &\|\tilde{\ve{x}}_{k+1}\|_{I_d-2\alpha\beta\Lambda^{\T}\Lambda}^2+\frac{\alpha}{\beta}\|\tilde{\ve{\lambda}}_{k+1}\|_{I_{Nm}-\gamma\mathcal{B}^2}^2
  +\frac{\alpha}{\beta\gamma}\|\tilde{\ve{y}}_{k+1}\|^2\notag\\
  &+\frac{1-\rho^2}{4L_3^2(1+\rho^2)}\|\ve{z}_{k+1}-\ve{1}_N\otimes\bar{z}_{k+1}\|^2\notag\\
  &+\frac{\alpha(1-\rho^2)}{\gamma(1+\rho^2)}\|\ve{\mu}_{k+1}-\ve{1}_N\otimes\bar{\mu}_{k+1}\|^2\notag\\
  \leq&\kappa\|\tilde{\ve{x}}_k\|^2-\alpha\nu\|\tilde{\ve{x}}_k\|^2-\alpha^2c_1\|\Lambda^{\T}\tilde{\ve{\lambda}}_k\|^2+\frac{\alpha}{\beta}\|\tilde{\ve{\lambda}}_k\|^2\notag\\
  &+2\frac{\alpha}{\beta}\|\mathcal{B}^2\tilde{\ve{\lambda}}_k\|^2
  -2\frac{\alpha}{\beta}\|\mathcal{B}\tilde{\ve{\lambda}}_k\|^2-\frac{\alpha}{\beta}\|\mathcal{B}\tilde{\ve{y}}_k\|^2+\frac{\alpha}{\beta\gamma}\|\tilde{\ve{y}}_k\|^2\notag\\
  &+\frac{1-\rho^2}{4L_3^2(1+\rho^2)}\kappa_1\|\ve{z}_k-\ve{1}_N\otimes\bar{z}_k\|^2\notag\\
  &+\frac{(1-\rho^2)\alpha}{(1+\rho^2)\gamma}\kappa_2\|\ve{\mu}_k-\ve{1}_N\otimes\bar{\mu}_k\|^2.
\end{align}

For the first two terms on the right-hand side of (\ref{error-x}), it gives
\begin{align}\label{eq:2}
  &\kappa\|\tilde{\ve{x}}_k\|^2-\alpha\nu\|\tilde{\ve{x}}_k\|^2\notag\\
  =&\kappa\|\tilde{\ve{x}}_k\|_{I_d-2\alpha\beta\Lambda^{\T}\Lambda}^2+\kappa\|\tilde{\ve{x}}_k\|_{2\alpha\beta\Lambda^{\T}\Lambda}^2-\alpha\nu\|\tilde{\ve{x}}_k\|^2\notag\\
  \leq&\kappa\|\tilde{\ve{x}}_k\|_{I_d-2\alpha\beta\Lambda^{\T}\Lambda}^2+(2\alpha\beta\kappa\lambda_{\max}(\Lambda^{\T}\Lambda)-\alpha\nu)\|\tilde{\ve{x}}_k\|^2\notag\\
  \leq&\kappa\|\tilde{\ve{x}}_k\|_{I_d-2\alpha\beta\Lambda^{\T}\Lambda}^2,
\end{align}
where the second inequality holds due to the selection of $\alpha$ and $\beta$ in (\ref{eq:stepsize}).

For the four terms with respect to $\tilde{\ve{\lambda}}_k$ on the right-hand side of (\ref{eq:1}), one has
\begin{align}\label{eq:3}
  &-\alpha^2c_1\|\Lambda^{\T}\tilde{\ve{\lambda}}_k\|^2+\frac{\alpha}{\beta}\|\tilde{\ve{\lambda}}_k\|^2
  +2\frac{\alpha}{\beta}\|\mathcal{B}^2\tilde{\ve{\lambda}}_k\|^2
  -2\frac{\alpha}{\beta}\|\mathcal{B}\tilde{\ve{\lambda}}_k\|^2\notag\\
  \leq&\frac{\alpha}{\beta}(1-\alpha\beta c_1\lambda_{\min}(\Lambda\Lambda^{\T}))\|\tilde{\ve{\lambda}}_k\|_{I_{Nm}-\gamma\mathcal{B}^2}^2\notag\\
  &+\frac{\alpha}{\beta}\gamma(1-\alpha\beta c_1\lambda_{\min}(\Lambda\Lambda^{\T}))\|\mathcal{B}\tilde{\ve{\lambda}}_k\|^2\notag\\
  &+2\frac{\alpha}{\beta}(\lambda_{\max}^2(\mathcal{B})-1)\|\mathcal{B}\tilde{\ve{\lambda}}_k\|^2\notag\\
  \leq&\frac{\alpha}{\beta}(1-\alpha\beta c_1\lambda_{\min}(\Lambda\Lambda^{\T}))\|\tilde{\ve{\lambda}}_k\|_{I_{Nm}-\gamma\mathcal{B}^2}^2,
\end{align}
where the second inequality lies on the selection of $\gamma$ in (\ref{eq:stepsize}).

For the two terms about $\tilde{\ve{y}}_k$ on the right-hand side of (\ref{error-x}), it can be derived that
\begin{align}\label{eq:4}
  -\frac{\alpha}{\beta}\|\mathcal{B}\tilde{\ve{y}}_k\|^2+\frac{\alpha}{\beta\gamma}\|\tilde{\ve{y}}_k\|^2\leq&\frac{\alpha}{\beta\gamma}(\|\tilde{\ve{y}}_k\|^2-\gamma\|\mathcal{B}\tilde{\ve{y}}_k\|^2)\notag\\
  \leq&\frac{\alpha}{\beta\gamma}(1-\gamma\underline{\sigma}^2(\mathcal{B}))\|\tilde{\ve{y}}_k\|^2,
\end{align}
in which the second inequality is based on $\|\mathcal{B}\tilde{\ve{y}}_k\|^2\geq \underline{\sigma}^2(\mathcal{B})\|\tilde{\ve{y}}_k\|^2$ under Assumption \ref{ass:network} since $\tilde{\ve{y}}_k$ lies in the range space of $\mathcal{B}$ \cite{Alghunaim2020Linear}.

Inserting (\ref{eq:2}), (\ref{eq:3}) and (\ref{eq:4}) into (\ref{eq:1}) leads to (\ref{thm:rate}). The proof is completed.
\end{proof}

\begin{rem}
Theorem \ref{thm:linearrate} shows the linear convergence result of the proposed Algorithm \ref{alg:primaldual}, and the precise convergence rate is presented with appropriate selection of constant stepsizes $\alpha$, $\beta$ and $\gamma$. We should point out that such stepsizes $\alpha$, $\beta$ and $\gamma$ satisfying (\ref{eq:stepsize}) can actually be taken. In fact, as long as $\alpha$ and $\frac{\alpha}{\gamma}$ are sufficiently small, the first inequality of (\ref{eq:stepsize}) holds, meanwhile, the second inequality can be ensured since $\rho<1$. Further, the third inequality holds by choosing small enough $\gamma$. The fourth inequality can be easily guaranteed with explicit expressions and the last inequality can also be assured as described below. Although in the first, second, third and last inequalities, $\alpha$ and $\gamma$ seem to be coupled together, which is slightly confusing, the above inequalities hold, for example, one can select $\gamma=\sqrt{\alpha}$ and take feasible $\alpha$ correspondingly. On the other hand, it can be seen that the upper bound of the convergence rate depends on not only the objective and constraint functions but also the communication network. Note that $\kappa_1$ and $\kappa_2$ rely on the network spectral gap $1-\rho$ \footnote{smaller network spectral gap $1-\rho$ means weaker network connectivity.} which leads to a high convergence rate of Algorithm \ref{alg:primaldual} with a large $1-\rho$ in view of $\tau$.
\end{rem}

\section{NUMERICAL SIMULATION}\label{sec4}
In this section, we conduct a numerical simulation to demonstrate the proposed algorithm. The experiment is performed using MATLAB R2022a on a laptop with AMD Ryzen 7 5800H CPU @ 3.20 GHz. Consider the following distributed quadratic optimization problem
\begin{equation*}
\begin{aligned}
  &\min_{x_1,\ldots,x_N}~\sum_{i=1}^N \|x_i-a_i\|^2+\|x_i-\frac{1}{N}\sum_{i=1}^Nx_i\|^2\\
  &\text{s.t.}~\sum_{i=1}^N A_ix_i\leq\sum_{i=1}^Nb_i.
\end{aligned}
\end{equation*}
During the implementation, set $N=60$, $m=n_i=5$, $A_i=I_m$, $i\in[N]$. For each $i\in[N]$, the entries of $a_i$ and $b_i$ are randomly chosen from $[1,3]$ and $[1,2]$ with uniform distributions, respectively. The above problem satisfies theoretical Assumptions \ref{ass:function} and \ref{ass:network}. The optimal solution $\ve{x}^*$ is estimated by the CVX toolbox.

To test the convergence performance of our proposed algorithm, first, we simulate Algorithm \ref{alg:primaldual} for different stepsizes, specifically, $\alpha=0.09$ and $\alpha=0.02$, respectively, $\beta=0.4$ and $\gamma=0.1$. The network graph is randomly generated as shown in Fig. \ref{fig_graph}. The relative error trajectories of $\frac{\|\ve{x}_k-\ve{x}^*\|}{\|\ve{x}^*\|}$ are presented in Fig. \ref{fig_error}. It can be seen that our algorithm converges linearly and increasing stepsizes leads to faster convergence speed. On the other hand, we consider different network topologies: the exponential graph, the random graph Fig. \ref{fig_graph} and the ring graph. The adjacency matrix related to the exponential graph is generated according to \cite{ying2021exp} and the one related to the ring graph is set as
$$ w_{ij}=\left\{
\begin{aligned}
0.5,  \quad & if\ i=j, \\
0.25,  \quad & if\ |i-j|=1, \\
0.25,  \quad & if\ (i,j)=(1,n)\ or\ (i,j) = (n,1), \\
0,     \quad & otherwise.
\end{aligned}
\right.
$$
The resulting $\rho$ of the adjacency matrices are $0.7143$, $0.9966$ and $0.9973$, respectively, which characterize the difference in the algebraic connectivity of the associated graphs. The experimental results are shown in Fig. \ref{fig_error_network}. Clearly, increasing connectivity of the network results in faster convergence speed, which is consistent with the obtained theoretical results.

\begin{figure}[htbp]
\centering
\includegraphics[width=7cm]{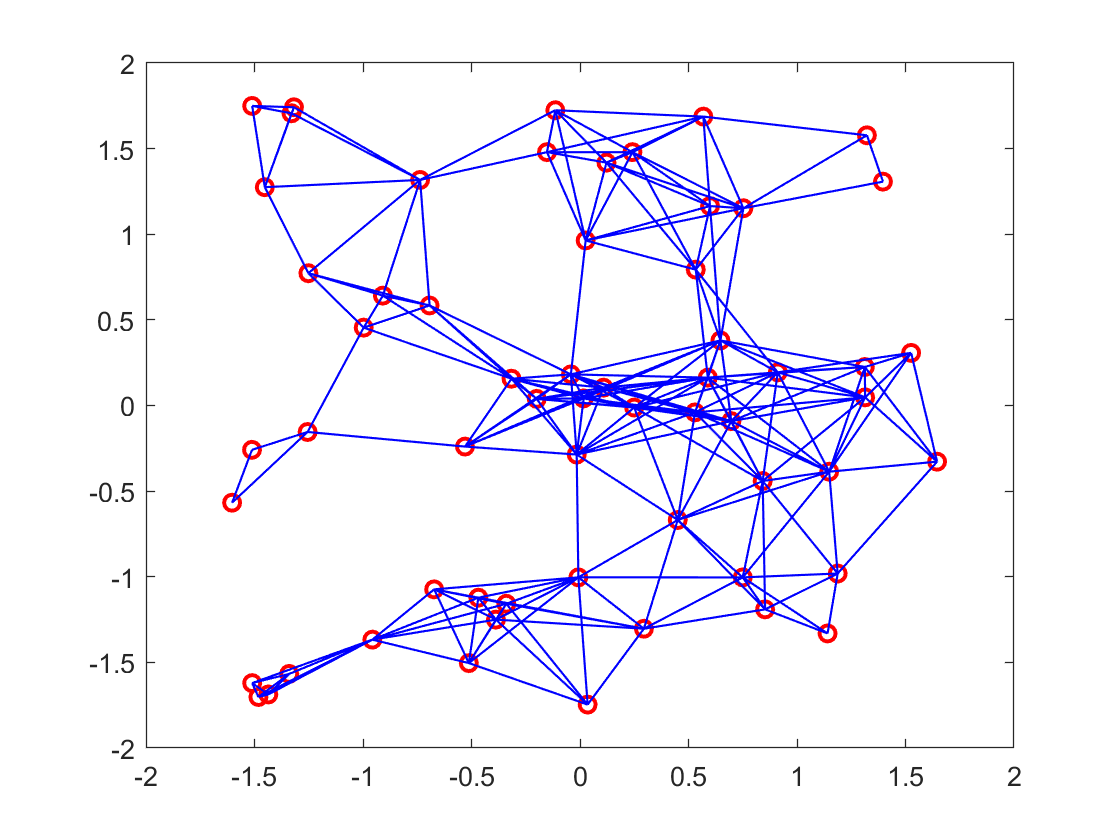}
\caption{Random network graph with 60 nodes.\label{fig_graph}}
\end{figure}
\begin{figure}[htbp]
\centering
\includegraphics[width=7cm]{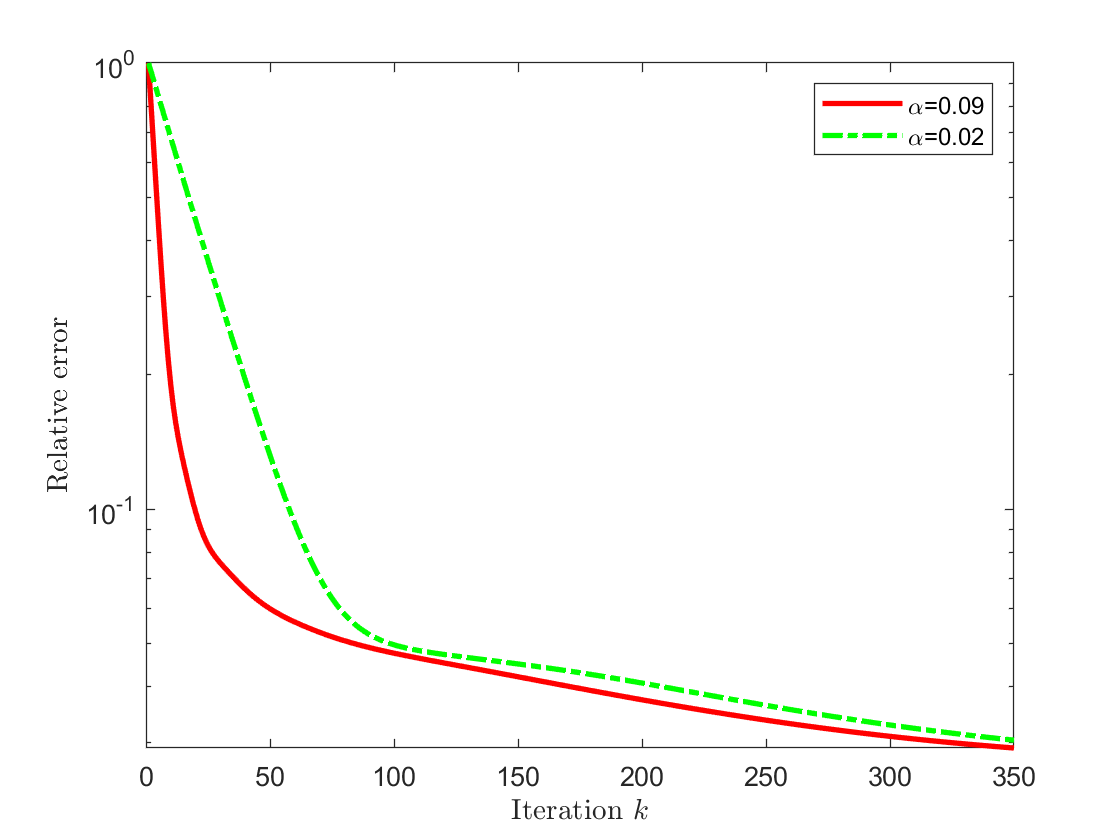}
\caption{The relative error trajectories of $\frac{\|\ve{x}_k-\ve{x}^*\|}{\|\ve{x}^*\|}$ under different stepsizes.\label{fig_error}}
\end{figure}
\begin{figure}[htbp]
\centering
\includegraphics[width=7cm]{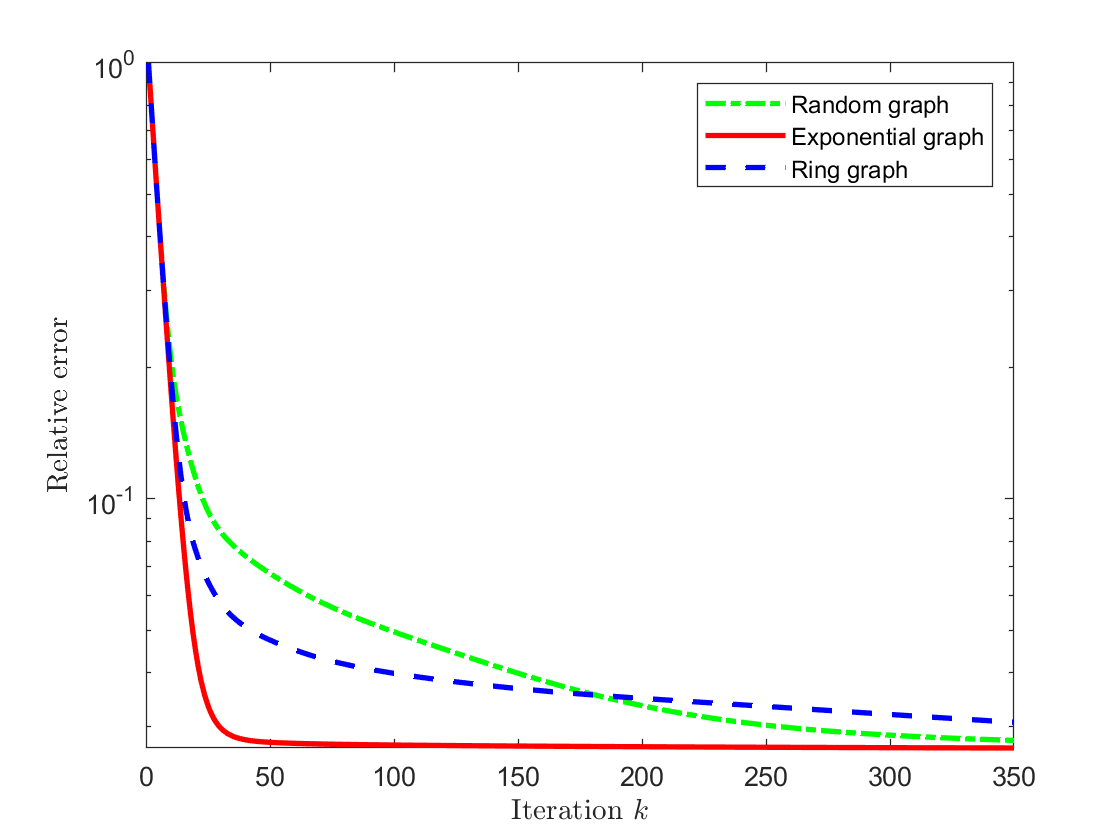}
\caption{The relative error trajectories of $\frac{\|\ve{x}_k-\ve{x}^*\|}{\|\ve{x}^*\|}$ over different network topologies.\label{fig_error_network}}
\end{figure}


\section{CONCLUSION}\label{sec5}
This work focuses on the distributed aggregative optimization with coupled affine inequality constraints. A novel aggregative primal-dual algorithm is proposed to address this problem, which utilizes the gradient tracking technique to handle the aggregate variable and the gradient average and the global optimal dual variable is introduced to deal with the inequality constraints. It has been proven that the proposed algorithm converges at a linear rate. In addition, explicit upper bounds on the stepsizes are provided. Finally, we provide a numerical simulation to support the theoretical results. In the future, it is a prospective direction to take the set constraints into account.

\appendices

\section{Proof of Lemma \ref{lem:average}}
\label{proof:lem:average}

Since $\mathcal{W}$ is doubly stochastic from Assumption \ref{ass:network}, multiplying $\frac{1}{N}\ve{1}_N^{\T}\otimes I_n$ on both sides of (\ref{alg:vec-2}) and (\ref{alg:vec-3}) gives that
\begin{align*}
  \bar{z}_{k+1}=&\bar{z}_k+\frac{1}{N}\sum_{i=1}^{N}h_i(x_{i,k+1})-\frac{1}{N}\sum_{i=1}^{N}h_i(x_{i,k}),\\
  \bar{\mu}_{k+1}=&\bar{\mu}_k+\frac{1}{N}\sum_{i=1}^{N}\nabla_2f_i(x_{i,k+1},z_{i,k+1})\notag\\
  &-\frac{1}{N}\sum_{i=1}^{N}\nabla_2f_i(x_{i,k},z_{i,k}).
\end{align*}

Through a simple recursion, it derives
\begin{align*}
  \bar{z}_k&=\bar{z}_0+\frac{1}{N}\sum_{i=1}^{N}h_i(x_{i,k})-\frac{1}{N}\sum_{i=1}^Nh_i(x_{i,0}),\\
  \bar{\mu}_k&=\bar{\mu}_0+\frac{1}{N}\sum_{i=1}^{N}\nabla_2f_i(x_{i,k},z_{i,k})-\frac{1}{N}\sum_{i=1}^N\nabla_2f_i(x_{i,0},z_{i,0}).
\end{align*}

Considering $z_{i,0}=h_i(x_{i,0})$ and $\mu_{i,0}=\nabla_2f_i(x_{i,0},z_{i,0})$, the assertions (\ref{lem:average-1}) and (\ref{lem:average-2}) hold.

\section{Proof of Lemma \ref{lem:fix}
\label{proof:lem:fix}}

Assume that $(\ve{x}^*,\lambda^*)$ is the KKT point satisfying (\ref{KKT-x}) and (\ref{KKT-dual}). Define $\ve{\lambda}^*=\ve{1}_N\otimes\lambda^*$,  $\ve{z}^*=\ve{1}_N\otimes \phi(\ve{x}^*)$, $\ve{\mu}^*=\ve{1}_N\otimes \frac{1}{N}\sum_{i=1}^N\nabla_2f_i(x_i^*,\phi(\ve{x}^*))$, then under Assumption \ref{ass:network}, $\mathcal{B}^2\ve{\lambda}^*=\ve{0}$, $\ve{z}^*$ and $\ve{\mu}^*$ satisfy (\ref{lem:fixed-pont-2}) and (\ref{lem:fixed-pont-3}), respectively, and (\ref{lem:fixed-pont-1}) holds due to (\ref{KKT-x}).
Define $\ve{v}^*=\ve{1}_N\otimes v^*=\ve{1}_N\otimes(\lambda^*+\frac{\beta}{N}(A\ve{x}^*-b))$, then (\ref{lem:fixed-pont-6}) holds. Note that (\ref{lem:fixed-pont-5}) is equivalent to $\lambda^*=\Pi_{\R^m_+}[v^*]$, that is,  $\frac{\beta}{N}(A\ve{x}^*-b)^{\T}(v-\lambda^*)\leq 0$, $\forall v\in \R^m_+$ by the definition of $\ve{v}^*$. Therefore, it follows from (\ref{KKT-dual}) that (\ref{lem:fixed-pont-5}) holds. Next, we prove that there exists $\ve{y}^*$ satisfying (\ref{lem:fixed-pont-4}). Since $(\ve{1}_N^{\T}\otimes I_m)(\ve{v}^*-\ve{\lambda}^*-\beta(\Lambda \ve{x}^*-\ve{b}))=\ve{0}$, then $\mathcal{B}\ve{y}^*$ is in the null space of $\ve{1}_N^{\T}\otimes I_m$. Then, it is also in the range space of $\mathcal{B}^2$. Therefore, there exists $\ve{y}^*$ in the range space such that (\ref{lem:fixed-pont-4}) holds.

Assume that $(\ve{x}^*,\ve{z}^*,\ve{\mu}^*,\ve{v}^*,\ve{y}^*,\ve{\lambda}^*)$ is a fixed point of (\ref{alg:vec}). It follows from (\ref{lem:fixed-pont-6}) and $\mathcal{B}^2\ve{\lambda}^*=\ve{0}$ that $\ve{v}^*=\ve{1}_N\otimes v^*$, $\ve{\lambda}^*=\ve{1}_N\otimes \lambda^*$ with $v^*,\lambda^*\in\R^m$. It can be obtained from (\ref{lem:fixed-pont-2}) and (\ref{lem:fixed-pont-3}) that $z_i^*=z_j^*$, $\mu_i^*=\mu_j^*$, $\forall i\neq j$, which incorporating with Lemma \ref{lem:average} yields
\begin{align*}
  z_i^*&=\frac{1}{N}\sum_{i=1}^N h_i(x_i^*)=\phi(\ve{x}^*),\notag\\
  \mu_i^*&=\frac{1}{N}\sum_{i=1}^N\nabla_2f_i(x_i^*,\phi(\ve{x}^*)),\ \forall i\in[N].
\end{align*}
Combining (\ref{lem:fixed-pont-2}) and (\ref{lem:fixed-pont-3}) with the above equalities, one has
\begin{equation}\label{eq:fix}
  \ve{z}^*=\ve{1}_N\otimes \phi(\ve{x}^*),\ \ve{\mu}^*=\ve{1}_N\otimes\frac{1}{N}\sum_{i=1}^N\nabla_2f_i(x_i^*,\phi(\ve{x}^*)).
\end{equation}
Substituting (\ref{eq:fix}) into (\ref{lem:fixed-pont-1}) leads to (\ref{KKT-x}).
Subsequently, by (\ref{lem:fixed-pont-4}) and (\ref{lem:fixed-pont-5}), if selecting $\ve{y}^*$ from the range space of $\mathcal{B}$, it can be concluded that $\ve{x}^*$ is the optimal solution to problem (\ref{prob}), and $\lambda^*$ is the optimal global dual variable. This ends the proof.

\section{Proof of Lemma \ref{lem:error-iter}}
\label{proof:lem:error-iter}

Firstly, note that
\begin{equation*}
  2(\tilde{\ve{x}}_k-\tilde{\ve{x}}_{k+1})^{\T}\tilde{\ve{x}}_{k+1}=\|\tilde{\ve{x}}_k\|^2-\|\tilde{\ve{x}}_{k+1}\|^2-\|\ve{x}_{k+1}-\ve{x}_k\|^2,
\end{equation*}
which incorporating with (\ref{alg:vec-error-1}) yields
\begin{align}\label{equation:iter-1}
  \|\tilde{\ve{x}}_{k+1}\|^2=&\|\tilde{\ve{x}}_k\|^2-\|\ve{x}_k-\ve{x}_{k+1}\|^2-2\alpha\tilde{\ve{\lambda}}_k^{\T}\Lambda\tilde{\ve{x}}_{k+1}\notag\\
  &-2\alpha\tilde{\ve{x}}_{k+1}^{\T}\Big(\nabla_1f(\ve{x}_k,\ve{z}_k)+\nabla h(\ve{x}_k)\ve{\mu}_k\notag\\
  &-\nabla_1f(\ve{x}^*,\ve{1}_N\otimes \phi(\ve{x}^*))\notag\\
  &-\nabla h(\ve{x}^*)\Big(\ve{1}_N\otimes\frac{1}{N}\sum_{i=1}^{N}\nabla_2f_i(x_i^*,\phi(\ve{x}^*))\Big)\Big).
\end{align}

Next, we deal with the last term on the right-hand side of (\ref{equation:iter-1}). It can be seen that
\begin{align}\label{equation:iter-2}
  &-2\alpha\tilde{\ve{x}}_{k+1}^{\T}\Big(\nabla_1f(\ve{x}_k,\ve{z}_k)+\nabla h(\ve{x}_k)\ve{\mu}_k-\nabla_1f(\ve{x}^*,\ve{1}_N\otimes \phi(\ve{x}^*))\notag\\
  &-\nabla h(\ve{x}^*)\Big(\ve{1}_N\otimes\frac{1}{N}\sum_{i=1}^{N}\nabla_2f_i(x_i^*,\phi(\ve{x}^*))\Big)\Big)\notag\\
  =&-2\alpha\tilde{\ve{x}}_k^{\T}\Big(\nabla_1f(\ve{x}_k,\ve{z}_k)+\nabla h(\ve{x}_k)\ve{\mu}_k-\nabla_1f(\ve{x}^*,\ve{1}_N\otimes \phi(\ve{x}^*))\notag\\
  &-\nabla h(\ve{x}^*)\Big(\ve{1}_N\otimes\frac{1}{N}\sum_{i=1}^{N}\nabla_2f_i(x_i^*,\phi(\ve{x}^*))\Big)\Big)\notag\\
  &+2\alpha(\ve{x}_k-\ve{x}_{k+1})^{\T}\Big(\nabla_1f(\ve{x}_k,\ve{z}_k)+\nabla h(\ve{x}_k)\ve{\mu}_k\notag\\
  &-\nabla_1f(\ve{x}^*,\ve{1}_N\otimes \phi(\ve{x}^*))\notag\\
  &-\nabla h(\ve{x}^*)\Big(\ve{1}_N\otimes\frac{1}{N}\sum_{i=1}^{N}\nabla_2f_i(x_i^*,\phi(\ve{x}^*))\Big)\Big).
\end{align}

For the second term on the right-hand side of (\ref{equation:iter-2}), one has
\begin{align}\label{equation:iter-3}
  &2\alpha(\ve{x}_k-\ve{x}_{k+1})^{\T}\Big(\nabla_1f(\ve{x}_k,\ve{z}_k)+\nabla h(\ve{x}_k)\ve{\mu}_k\notag\\
  &-\nabla_1f(\ve{x}^*,\ve{1}_N\otimes \phi(\ve{x}^*))\notag\\
  &-\nabla h(\ve{x}^*)\Big(\ve{1}_N\otimes\frac{1}{N}\sum_{i=1}^{N}\nabla_2f_i(x_i^*,\phi(\ve{x}^*))\Big)\Big)\notag\\
  =&-\Big\|\ve{x}_k-\ve{x}_{k+1}-\alpha\Big(\nabla_1f(\ve{x}_k,\ve{z}_k)+\nabla h(\ve{x}_k)\ve{\mu}_k\notag\\
  &-\nabla_1f(\ve{x}^*,\ve{1}_N\otimes \phi(\ve{x}^*))\notag\\
  &-\nabla h(\ve{x}^*)\Big(\ve{1}_N\otimes\frac{1}{N}\sum_{i=1}^{N}\nabla_2f_i(x_i^*,\phi(\ve{x}^*))\Big)\Big)\Big\|^2\notag\\
  &+\|\ve{x}_k-\ve{x}_{k+1}\|^2+\alpha^2\Big\|\nabla_1f(\ve{x}_k,\ve{z}_k)+\nabla h(\ve{x}_k)\ve{\mu}_k\notag\\
  &-\nabla_1f(\ve{x}^*,\ve{1}_N\otimes \phi(\ve{x}^*))\notag\\
  &-\nabla h(\ve{x}^*)\Big(\ve{1}_N\otimes\frac{1}{N}\sum_{i=1}^{N}\nabla_2f_i(x_i^*,\phi(\ve{x}^*))\Big)\Big\|^2\notag\\
  =&-\alpha^2\|\Lambda^{\T}\tilde{\ve{\lambda}}_k\|^2+\|\ve{x}_k-\ve{x}_{k+1}\|^2+\alpha^2\Big\|\nabla_1f(\ve{x}_k,\ve{z}_k)\notag\\
  &+\nabla h(\ve{x}_k)\ve{\mu}_k\notag-\nabla_1f(\ve{x}^*,\ve{1}_N\otimes \phi(\ve{x}^*))\notag\\
  &-\nabla h(\ve{x}^*)\Big(\ve{1}_N\otimes\frac{1}{N}\sum_{i=1}^{N}\nabla_2f_i(x_i^*,\phi(\ve{x}^*))\Big)\Big\|^2,
\end{align}
where iteration (\ref{alg:vec-error-1}) has been used to obtain the last equality.

Adding and subtracting the term $\nabla h(\ve{x}_k)(\ve{1}_N\otimes\bar{\mu}_k)$, and utilizing $\|a+b\|^2\leq2\|a\|^2+2\|b\|^2$ for any $a, b\in \R^d$, it can be obtained that
\begin{align*}
  &\Big\|\nabla_1f(\ve{x}_k,\ve{z}_k)+\nabla h(\ve{x}_k)\ve{\mu}_k-\nabla_1f(\ve{x}^*,\ve{1}_N\otimes \phi(\ve{x}^*))\notag\\
  &-\nabla h(\ve{x}^*)\Big(\ve{1}_N\otimes\frac{1}{N}\sum_{i=1}^{N}\nabla_2f_i(x_i^*,\phi(\ve{x}^*))\Big)\Big\|^2\notag\\
  \leq&2\Big\|\nabla_1f(\ve{x}_k,\ve{z}_k)+\nabla h(\ve{x}_k)\ve{1}_N\otimes\bar{\mu}_k-\nabla_1f(\ve{x}^*,\ve{1}_N\otimes \phi(\ve{x}^*))\notag\\
  &-\nabla h(\ve{x}^*)\Big(\ve{1}_N\otimes\frac{1}{N}\sum_{i=1}^{N}\nabla_2f_i(x_i^*,\phi(\ve{x}^*))\Big)\Big\|^2\notag\\
  &+2\|\nabla h(\ve{x}_k)(\ve{\mu}_k-\ve{1}_N\otimes\bar{\mu}_k)\|^2\notag\\
  \leq&4L_1^2(\|\tilde{\ve{x}}_k\|^2+\|\ve{z}_k-\ve{1}_N\otimes\phi(\ve{x}^*)\|^2))+2L_3^2\|\ve{\mu}_k-\ve{1}_N\otimes\bar{\mu}_k\|^2\notag\\
  \leq&4L_1^2\|\tilde{\ve{x}}_k\|^2+8L_1^2\|\ve{z}_k-\ve{1}_N\otimes\bar{z}_k\|^2\notag\\
  &+8L_1^2\|\ve{1}_N\otimes(\bar{z}_k-\phi(\ve{x}^*))\|^2+2L_3^2\|\ve{\mu}_k-\ve{1}_N\otimes\bar{\mu}_k\|^2,
\end{align*}
where the relation $\bar{\mu}_k=\frac{1}{N}\sum_{i=1}^N\nabla_2f_i(x_{i,k},z_{i,k})$, Assumption \ref{ass:function}-ii) and \ref{ass:function}-iv) have been employed in the second inequality. For the terms $\|\ve{1}_N\otimes(\bar{z}_k-\phi(\ve{x}^*))\|^2$ on the right-hand side of the above formula, in view of (\ref{lem:average-1}), it holds that
\begin{align*}
  \|\ve{1}_N\otimes(\bar{z}_k-\phi(\ve{x}^*))\|^2&=N\Big\|\frac{1}{N}\sum_{i=1}^N(h_i(x_{i,k})-h_i(x_i^*))\Big\|^2\notag\\
  &\leq\frac{1}{N}\Big(\sum_{i=1}^N\|h_i(x_{i,k})-h_i(x_i^*)\|\Big)^2\notag\\
  &\leq\frac{1}{N}\Big(\sum_{i=1}^NL_3\|x_{i,k}-x_i^*\|\Big)^2\notag\\
  &\leq L_3^2\sum_{i=1}^N\|x_{i,k}-x_i^*\|^2\notag\\
  &=L_3^2\|\tilde{\ve{x}}_k\|^2,
\end{align*}
where Assumption \ref{ass:function}-iv) has been exploited in the second inequality, and the C-R inequality $\|\sum_{i=1}^pa_i\|^2\leq p\sum_{i=1}^p\|a_i\|^2$ for any $a_i\in \R^d$, $\forall i\in[p]$ has been leveraged in the first and last inequalities. Hence, combining the above two inequalities, it can be derived that
\begin{align}\label{equation:iter-4}
  &\Big \|\nabla_1f(\ve{x}_k,\ve{z}_k)+\nabla h(\ve{x}_k)\ve{\mu}_k-\nabla_1f(\ve{x}^*,\ve{1}_N\otimes \phi(\ve{x}^*))\notag\\
  &-\nabla h(\ve{x}^*)\Big(\ve{1}_N\otimes\frac{1}{N}\sum_{i=1}^{N}\nabla_2f_i(x_i^*,\phi(\ve{x}^*))\Big)\Big\|^2\notag\\
  \leq&4L_1^2(1+2L_3^2)\|\tilde{\ve{x}}_k\|^2+8L_1^2\|\ve{z}_k-\ve{1}_N\otimes\bar{\ve{z}}_k\|^2\notag\\
  &+2L_3^2\|\ve{\mu}_k-\ve{1}_N\otimes\bar{\mu}_k\|^2.
\end{align}

Proceeding in a similar way to (\ref{equation:iter-4}), one has
\begin{align}\label{equation:iter-5}
  &\Big\|\nabla_1f(\ve{x}_k,\ve{z}_k)+\nabla h(\ve{x}_k)\ve{\mu}_k-\nabla_1f(\ve{x}_k,\ve{1}_N\otimes\phi(\ve{x}_k))\notag\\
  &-\nabla h(\ve{x}_k)\Big(\ve{1}_N\otimes\frac{1}{N}\nabla_2f_i(x_{i,k},\phi(\ve{x}_k))\Big)\Big\|^2\notag\\
  \leq&2L_1^2\|\ve{z}_k-\ve{1}_N\otimes\bar{z}_k\|^2+2L_3^2\|\ve{\mu}_k-\ve{1}_N\otimes\bar{\mu}_k\|^2.
\end{align}
For the first term on the right-hand side of (\ref{equation:iter-2}), note that
\begin{align}\label{equation:iter-6}
  &-2\alpha\tilde{\ve{x}}_k^{\T}\Big(\nabla_1f(\ve{x}_k,\ve{z}_k)+\nabla h(\ve{x}_k)\ve{\mu}_k-\nabla_1f(\ve{x}^*,\ve{1}_N\otimes \phi(\ve{x}^*))\notag\\
  &-\nabla h(\ve{x}^*)\Big(\ve{1}_N\otimes\frac{1}{N}\sum_{i=1}^{N}\nabla_2f_i(x_i^*,\phi(\ve{x}^*))\Big)\Big)\notag\\
  =&-2\alpha\tilde{\ve{x}}_k^{\T}\Big(\nabla_1f(\ve{x}_k,\ve{z}_k)+\nabla h(\ve{x}_k)\ve{\mu}_k-\nabla_1f(\ve{x}_k,\ve{1}_N\otimes\phi(\ve{x}_k))\notag\\
  &-\nabla h(\ve{x}_k)\Big(\ve{1}_N\otimes\frac{1}{N}\nabla_2f_i(x_{i,k},\phi(\ve{x}_k))\Big)\Big)\notag\\
  &-2\alpha\tilde{\ve{x}}_k^{\T}\Big(\nabla_1f(\ve{x}_k,\ve{1}_N\otimes\phi(\ve{x}_k))-\nabla_1f(\ve{x}^*,\ve{1}_N\otimes \phi(\ve{x}^*))\notag\\
  &+\nabla h(\ve{x}_k)\Big(\ve{1}_N\otimes\frac{1}{N}\nabla_2f_i(x_{i,k},\phi(\ve{x}_k))\Big)\notag\\
  &-\nabla h(\ve{x}^*)\Big(\ve{1}_N\otimes\frac{1}{N}\sum_{i=1}^{N}\nabla_2f_i(x_i^*,\phi(\ve{x}^*))\Big)\Big)\notag\\
  \leq&\frac{\alpha\nu}{2}\|\tilde{\ve{x}}_k\|^2+\frac{2\alpha}{\nu}(2L_1^2\|\ve{z}_k-\ve{1}_N\otimes\bar{z}_k\|^2\notag\\
  &+2L_3^2\|\ve{\mu}_k-\ve{1}_N\otimes\bar{\mu}_k\|^2)-2\nu\alpha\|\tilde{\ve{x}}_k\|^2,
\end{align}
where the Young's inequality $a^{\T}b\leq\frac{\iota}{2}\|a\|^2+\frac{1}{2\iota}\|b\|^2$ for $\iota=\frac{\nu}{2}$, Assumption \ref{ass:function}-i) and (\ref{equation:iter-5}) have been used in the last inequality.

Combining (\ref{equation:iter-1}), (\ref{equation:iter-2}), (\ref{equation:iter-3}), (\ref{equation:iter-4}) and (\ref{equation:iter-6}), the assertion (\ref{error-x}) holds.

Utilizing the same analytical approach to \cite[Lemma 2]{meng2022linear}, the assertion (\ref{error-dual}) holds.

From (\ref{alg:vec-error-1}), it follows from (\ref{equation:iter-4}) that
\begin{align}\label{equation:iter-11}
  &\|\ve{x}_{k+1}-\ve{x}_k\|^2\notag\\
  \leq&2\alpha^2\Big\|\nabla_1f(\ve{x}_k,\ve{z}_k)+\nabla h(\ve{x}_k)\ve{\mu}_k-\nabla_1f(\ve{x}^*,\ve{1}_N\otimes \phi(\ve{x}^*))\notag\\
  &-\nabla h(\ve{x}^*)\Big(\ve{1}_N\otimes\frac{1}{N}\sum_{i=1}^{N}\nabla_2f_i(x_i^*,\phi(\ve{x}^*))\Big)\Big\|^2+2\alpha^2\|\Lambda^{\T}\tilde{\ve{\lambda}}_k\|^2\notag\\
  \leq&8L_1^2(1+2L_3^2)\alpha^2\|\tilde{\ve{x}}_k\|^2+16L_1^2\alpha^2\|\ve{z}_k-\ve{1}_N\otimes\bar{\ve{z}}_k\|^2\notag\\
  &+4L_3^2\alpha^2\|\ve{\mu}_k-\ve{1}_N\otimes\bar{\mu}_k\|^2+2\alpha^2\|\Lambda^{\T}\tilde{\ve{\lambda}}_k\|^2.
\end{align}
In view of (\ref{alg:vec-2}), it holds that
\begin{align*}
  &\|\ve{z}_{k+1}-\ve{1}_N\otimes\bar{z}_{k+1}\|^2\notag\\
  =&\|(\mathcal{W}\otimes I_n)\ve{z}_k-(\frac{1}{N}\ve{1}_N\ve{1}_N^{\T}\otimes I_n)\ve{z}_k\notag\\
  &+(I_{Nn}-\frac{1}{N}\ve{1}_N\ve{1}_N^{\T}\otimes I_n)(h(\ve{x}_{k+1})-h(\ve{x}_k))\|^2\notag\\
  \leq&\frac{1+\rho^2}{2}\|(\mathcal{W}\otimes I_n)\ve{z}_k-(\frac{1}{N}\ve{1}_N\ve{1}_N^{\T}\otimes I_n)\ve{z}_k\|^2\notag\\
  &+\frac{1+\rho^2}{1-\rho^2}\|(I_{Nn}-\frac{1}{N}\ve{1}_N\ve{1}_N^{\T}\otimes I_n)(h(\ve{x}_{k+1})-h(\ve{x}_k))\|^2\notag\\
  \leq&\frac{1+\rho^2}{2}\|\ve{z}_k-\ve{1}_N\otimes\bar{z}_k\|^2+\frac{1+\rho^2}{1-\rho^2}L_3^2\|\ve{x}_{k+1}-\ve{x}_k\|^2,
\end{align*}
where Lemma \ref{lem:network}-i) has been exploited in the first equality, the first inequality follows from $\|a+b\|^2\leq(1+\iota)\|a\|^2+(1+\frac{1}{\iota})\|b\|^2$ for $\iota=\frac{1-\rho^2}{2\rho^2}$, and the last inequality holds based on Lemma \ref{lem:network}-iv), Assumption \ref{ass:function}-iv) and the fact $\|I_{Nn}-\frac{1}{N}\ve{1}_N\ve{1}_N^{\T}\otimes I_n\|=1$. Combining the above inequality and (\ref{equation:iter-11}) yields (\ref{error-aggre}).

In light of $\|\mathcal{W}-I_N\|\leq 2$ from Lemma \ref{lem:network}-iii) and Assumption \ref{ass:function}-iv), it can be derived from (\ref{alg:vec-2}) that
\begin{align}\label{equation:iter-12}
  &\|\ve{z}_{k+1}-\ve{z}_k\|^2\notag\\
  =&\|(\mathcal{W}\otimes I_n-I_N\otimes I_n)(\ve{z}_k-\ve{1}_N\otimes\bar{z}_k)+h(\ve{x}_{k+1})-h(\ve{x}_k)\|^2\notag\\
  \leq&2\|\mathcal{W}-I_N\|^2\|\ve{z}_k-\ve{1}_N\otimes\bar{z}_k\|^2+2L_3^2\|\ve{x}_{k+1}-\ve{x}_k\|^2\notag\\
  \leq&8\|\ve{z}_k-\ve{1}_N\otimes\bar{z}_k\|^2+2L_3^2\|\ve{x}_{k+1}-\ve{x}_k\|^2.
\end{align}
Similarly, it follows from (\ref{alg:vec-3}) that
\begin{align*}
  &\|\ve{\mu}_{k+1}-\ve{1}_N\otimes\bar{\mu}_{k+1}\|^2\notag\\
  \leq&\frac{1+\rho^2}{2}\|\ve{\mu}_k-\ve{1}_N\otimes\bar{\mu}_k\|^2\notag\\
  &+\frac{1+\rho^2}{1-\rho^2}\|\nabla_2f(\ve{x}_{k+1},\ve{z}_{k+1})-\nabla_2f(\ve{x}_k,\ve{z}_k)\|^2\notag\\
  \leq&\frac{1+\rho^2}{2}\|\ve{\mu}_k-\ve{1}_N\otimes\bar{\mu}_k\|^2\notag\\
  &+2\frac{1+\rho^2}{1-\rho^2}L_2^2(\|\ve{x}_{k+1}-\ve{x}_k\|^2+\|\ve{z}_{k+1}-\ve{z}_k\|^2),
\end{align*}
which incorporating with (\ref{equation:iter-11}) and (\ref{equation:iter-12}) yields relation (\ref{error-nabla2}).


\begin{thebibliography}{00}

\bibitem{kansal2013optimal}
S.~Kansal, V.~Kumar, and B.~Tyagi, ``Optimal placement of different type of
  {DG} sources in distribution networks,'' \emph{Int. J. Electr. Power Energy Syst.}, vol.~53, pp. 752--760, Dec. 2013.

\bibitem{barrera2014dynamic}
J.~Barrera and A.~Garcia, ``Dynamic incentives for congestion control,''
  \emph{IEEE Trans. Autom. Control}, vol.~60, no.~2, pp. 299--310, Feb.
  2015.

\bibitem{cao2021distributed}
X.~Cao and K.~R. Liu, ``Distributed newton's method for network cost
  minimization,'' \emph{IEEE Trans. Autom. Control}, vol.~66,
  no.~3, pp. 1278--1285, Mar. 2021.

\bibitem{Li2022Aggregative}
X.~Li, L.~Xie, and Y.~Hong, ``Distributed aggregative optimization over
  multi-agent networks,'' \emph{IEEE Trans. Autom. Control},
  vol.~67, no.~6, pp. 3165--3171, Jun. 2022.

\bibitem{Li2022DistributedOnline}
X.~Li, X.~Yi, and L.~Xie, ``Distributed online convex optimization with an
  aggregative variable,'' \emph{IEEE Trans. Control Network Syst.}, vol.~9, no.~1, pp. 438--449, Mar. 2022.

\bibitem{Chen&LiangDistributed}
Z.~Chen and S.~Liang, ``Distributed aggregative optimization with quantized
  communication,'' \emph{Kybernetika}, vol.~58, no.~1, pp. 123--144, 2022.

\bibitem{Carnevale2023nonconvex}
G.~Carnevale, N.~Mimmo, and G.~Notarstefano, ``Nonconvex distributed feedback
  optimization for aggregative cooperative robotics,'' \emph{arXiv preprint
  arXiv:2302.01892}, 2023.

\bibitem{Wang2022Distributed}
T.~Wang and P.~Yi, ``Distributed projection-free algorithm for constrained
  aggregative optimization,'' \emph{arXiv preprint arXiv:2207.11885}, 2022.

\bibitem{liu2023acc}
J.~Liu, S.~Chen, S.~Cai, and C.~Xu, ``Accelerated distributed aggregative
  optimization,'' \emph{arXiv preprint arXiv:2304.08051}, 2023.

\bibitem{Koshal2016Distributed}
J.~Koshal, A.~Nedic, and U.~V. Shanbhag, ``Distributed algorithms for
  aggregative games on graphs,'' \emph{Oper. Res.}, vol.~64, no.~3,
  pp. 680--704, 2016.

\bibitem{Liang2017Distributednash}
S.~Liang, P.~Yi, and Y.~Hong, ``Distributed {N}ash equilibrium seeking for
  aggregative games with coupled constraints,'' \emph{Automatica}, vol.~85, pp.
  179--185, Nov. 2017.

\bibitem{gadjov2018passivity}
D.~Gadjov and L.~Pavel, ``A passivity-based approach to {N}ash equilibrium
  seeking over networks,'' \emph{IEEE Trans. Autom. Control},
  vol.~64, no.~3, pp. 1077--1092, Mar. 2019.

\bibitem{yi2019operator}
P.~Yi and L.~Pavel, ``An operator splitting approach for distributed
  generalized {N}ash equilibria computation,'' \emph{Automatica}, vol. 102, pp.
  111--121, Apr. 2019.

\bibitem{Xu2019Dual}
J.~Xu, S.~Zhu, Y.~C. Soh, and L.~Xie, ``A dual splitting approach for
  distributed resource allocation with regularization,'' \emph{IEEE Trans. Control Network Syst.}, vol.~6, no.~1, pp. 403--414, Mar.
  2019.

\bibitem{camisa2021distributed}
A.~Camisa, F.~Farina, I.~Notarnicola, and G.~Notarstefano, ``Distributed
  constraint-coupled optimization via primal decomposition over random
  time-varying graphs,'' \emph{Automatica}, vol. 131, p. 109739, Sep. 2021.

\bibitem{falsone2017dual}
A.~Falsone, K.~Margellos, S.~Garatti, and M.~Prandini, ``Dual decomposition for
  multi-agent distributed optimization with coupling constraints,''
  \emph{Automatica}, vol.~84, pp. 149--158, Oct. 2017.

\bibitem{heydaribeni2019distributed}
N.~Heydaribeni and A.~Anastasopoulos, ``Distributed mechanism design for
  network resource allocation problems,'' \emph{IEEE Trans. Network Sci. Eng.}, vol.~7, no.~2, pp. 621--636, 2020.

\bibitem{guo2022exponential}
L.~Guo, X.~Shi, J.~Cao, and Z.~Wang, ``Exponential convergence of primal-dual
  dynamics under general conditions and its application to distributed
  optimization,'' \emph{IEEE Trans. Neural Networks Learn. Syst.}, 2022.

\bibitem{li2020distributed}
X.~Li, G.~Feng, and L.~Xie, ``Distributed proximal algorithms for multiagent
  optimization with coupled inequality constraints,'' \emph{IEEE Trans. Autom. Control}, vol.~66, no.~3, pp. 1223--1230, Mar. 2021.

\bibitem{alghunaim2021dual}
S.~A. Alghunaim, Q.~Lyu, M.~Yan, and A.~H. Sayed, ``Dual consensus proximal
  algorithm for multi-agent sharing problems,'' \emph{IEEE Trans. Signal Process.}, vol.~69, pp. 5568--5579, 2021.

\bibitem{Chang2015Multi-Agent}
T.-H. Chang, M.~Hong, and X.~Wang, ``Multi-agent distributed optimization via
  inexact consensus {ADMM},'' \emph{IEEE Trans. Signal Process.},
  vol.~63, no.~2, pp. 482--497, Jan. 2015.

\bibitem{le2017neurodynamic}
X.~Le, S.~Chen, Z.~Yan, and J.~Xi, ``A neurodynamic approach to distributed
  optimization with globally coupled constraints,'' \emph{IEEE Trans. Cybern.}, vol.~48, no.~11, pp. 3149--3158, Nov. 2018.

\bibitem{Li2023Proximal}
J.~Li, Q.~An, and H.~Su, ``Proximal nested primal-dual gradient algorithms for
  distributed constraint-coupled composite optimization,'' \emph{Appl. Math. Comput.}, vol. 444, 2023.

\bibitem{car2022distributed}
G.~Carnevale, A.~Camisa, and G.~Notarstefano, ``Distributed online aggregative
  optimization for dynamic multi-robot coordination,'' \emph{IEEE Trans. Autom. Control}, 2022.

\bibitem{huang2021primal}
Y.~Huang and J.~Hu, ``A primal decomposition approach to globally coupled
  aggregative optimization over networks,'' in \emph{Proc. 60th IEEE Conf.
  Decision Control}, 2021, pp. 3830--3835.

\bibitem{alghunaim2020proximal}
S.~A. Alghunaim, K.~Yuan, and A.~H. Sayed, ``A proximal diffusion strategy for
  multiagent optimization with sparse affine constraints,'' \emph{IEEE Trans. Autom. Control}, vol.~65, no.~11, pp. 4554--4567, Nov. 2020.

\bibitem{meng2022linear}
M.~Meng and X.~Li, ``Linear last-iterate convergence for continuous games with
  coupled inequality constraints,'' \emph{arXiv preprint arXiv:2207.13924},
  2022.

\bibitem{horn2012MA}
R.~A. Horn and C.~R. Johnson, \emph{Matrix Analysis}. Cambridge, U.K.: Cambridge Univ. Press, 2012.

\bibitem{pu2018}
S.~Pu, W.~Shi, J.~Xu, and A.~Nedi\'{c}, ``A push-pull gradient method for
  distributed optimization in networks,'' in \emph{Proc. 57th IEEE Conf.
  Decision Control}, 2018, pp. 3385--3390.

\bibitem{Alghunaim2020Linear}
S.~A. Alghunaim and A.~H. Sayed, ``Linear convergence of primal-dual gradient
  methods and their performance in distributed optimization,''
  \emph{Automatica}, vol. 117, Jul. 2020.

\bibitem{ying2021exp}
B.~Ying, K.~Yuan, Y.~Chen, H.~Hu, P.~Pan, and W.~Yin, ``Exponential graph is
  provably efficient for decentralized deep training,'' in \emph{Proc. Adv. Neural Inf. Process Syst.}, vol.~17, 2021, pp. 13\,975--13\,987.
\end{thebibliography}

\begin{IEEEbiography}[{\includegraphics[width=1in,height=1.25in,clip,keepaspectratio]{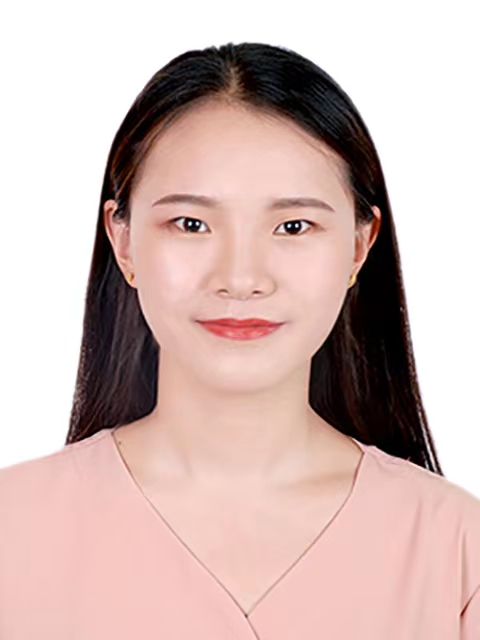}}]
{Kaixin Du}
received the B.S. degree in mathematics and applied mathematics from Henan University, Kaifeng, China in 2019, and received the M.S. degree in operations research and cybernetics from Dalian University of Technology, Dalian, China in 2022. She is currently pursuing the Ph.D. degree in intelligent science and technology from Shanghai Research Institute for Intelligent Autonomous Systems, Tongji University, Shanghai, China.

Her current research interests include stochastic optimization and distributed optimization.

\end{IEEEbiography}
\begin{IEEEbiography}[{\includegraphics[width=1in,height=1.25in,clip,keepaspectratio]{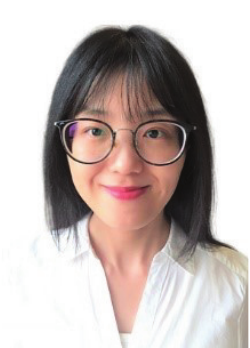}}]
{Min Meng}
received the B.S. and Ph.D. degrees from Shandong University, China, in 2010 and 2015, respectively. She had a position as a Research Associate in Department of Mechanical Engineering, The University of Hong Kong, Hong Kong, China, from April to October in 2014, from July to September in 2016, and from January to March in 2017. From July 2015 to June 2016, she was a Research Associate in Department of Biomedical Engineering, City University of Hong Kong, Hong Kong, China. From July 2017 to September 2020, she worked as a Research Fellow in the School of Electrical and Electronic Engineering, Nanyang Technological University, Singapore. In September 2020, she joined the Department of Control Science and Engineering, Tongji University, and Shanghai Research Institute for Intelligent Autonomous Systems, Tongji University, Shanghai, China, where she is a professor now.

Her research interests include multi-agent systems, distributed games and optimization, Boolean networks, distributed secure control and estimation, especially for large-scale networks, cooperative control of multiple agents, as well as applications to autonomous vehicles and robotics, genetic regulatory networks, and wireless sensor networks, etc.
\end{IEEEbiography}

\end{document}